\documentclass[a4paper,10pt,reqno]{amsart}
\usepackage[body={13cm,21cm}]{geometry}
\usepackage[initials]{amsrefs}

\usepackage{amssymb,amscd}
\usepackage[all]{xy}

\usepackage[mathscr]{eucal}
\usepackage{enumerate}

\numberwithin{equation}{section}
\theoremstyle{plain}
\newtheorem{thm}[equation]{Theorem}
\newtheorem{cor}[equation]{Corollary}
\newtheorem{lem}[equation]{Lemma}
\newtheorem{prop}[equation]{Proposition}

\newtheorem{definition}[equation]{Definition}
\newtheorem{exa}[equation]{Example}
\newtheorem{rem}[equation]{Remark}

\theoremstyle{definition}

\theoremstyle{remark}

\providecommand{\abs}[1]{\lvert#1\rvert}

\DeclareMathOperator{\Gal}{Gal}

\def\Q{{\mathbb Q}}
\def\Z{{\mathbb Z}}

\def\G{{\mathbb G}}

\def\A{{\mathbb A}}
\def\Gal{{\rm Gal}}

\DeclareFontFamily{U}{wncy}{}
\DeclareFontShape{U}{wncy}{m}{n}{%
<5>wncyr5%
<6>wncyr6%
<7>wncyr7%
<8>wncyr8%
<9>wncyr9%
<10>wncyr10%
<11>wncyr10%
<12>wncyr6%
<14>wncyr7%
<17>wncyr8%
<20>wncyr10%
<25>wncyr10}{}
\DeclareMathAlphabet{\cyr}{U}{wncy}{m}{n}
\begin{document}

\title[Integral Points for Groups of Multiplicative Type]
{Integral Points for Groups of Multiplicative Type}

\author{Dasheng Wei$^1$}
\author{Fei Xu$^{1,2}$}

\address{$^1$ Academy of Mathematics and System Science,  CAS, Beijing
100190, P.R.China}

\email{dshwei@amss.ac.cn}

\address{$^2$ School of Mathematics, Capital Normal University,
Beijing 100048, P.R.China}

\email{xufei@math.ac.cn}

\date{\today}

\maketitle

To Professor Kezheng Li on his 65-th birthday

 \bigskip

\section*{\it Abstract}

We construct a finite subgroup of Brauer-Manin obstruction for
detecting the existence of integral points on integral models of
homogeneous spaces of linear algebraic groups of multiplicative
type. As application, the strong approximation theorem for linear
algebraic groups of multiplicative type is established. Moreover,
the sum of two integral squares over some quadratic fields is
discussed.

\bigskip

{\it MSC classification} : 11E72, 11G35, 11R37, 14F22, 14G25, 20G30,
20G35


\bigskip

{\it Keywords} :  integral point, linear algebraic group of
multiplicative type, Galois cohomology, Brauer-Manin obstruction,
strong approximation, sum of two squares.

\section*{Introduction} \label{sec.notation}

The integral points on homogeneous spaces of semi-simple and simply
connected linear algebraic groups of non-compact type were studied
by Borovoi and Rudnick in \cite{BR95} and homogeneous spaces of
connected semi-simple linear algebraic groups of non-compact type by
Colliot-Th\'el\`ene and the second named author in \cite{CTX} using
the strong approximation theorem and the Brauer-Manin obstruction.
Recently, Harari \cite{Ha08} showed that the Brauer-Manin
obstruction accounts for the nonexistence of integral points.
Colliot-Th\'el\`ene noticed that a finite subgroup of the Brauer
group is enough to account for the nonexistence of integral points
by the compactness arguments. These results are nonconstructive:
they do not say which finite subgroup to use. In our previous paper
\cite{WX}, we construct such finite groups for multi-norm tori
explicitly. In this paper, we'll extend this construction for linear
algebraic groups of multiplicative type.

The paper is organized as follows. In Section 1, we'll give an
explicit construction of the so called admissible groups (see
Definition \ref{adm}) and establish the criterion for existence of
the integral points in terms of admissible groups. In Section 2, we
translate such admissible groups into the finite Brauer-Manin
obstruction. As an application, we establish the strong
approximation theorem for the group of multiplicative type in
Section 3. In Section 4, we discussed the sum of two integral
squares over some quadratic fields.

Notation and terminology are standard if not explained. Let $F$ be a
number field, $o_F$ be the ring of integers of $F$, $\Omega_F$
be the set of all primes in $F$ and $\infty_F$ be the set all
infinite primes in $F$. For simplicity, we write $p <\infty_F$
for $p \in \Omega_F\setminus \infty_{F}$. For any finite set
$S_0 \subset \Omega_F\setminus \infty_F$, the ring of $S_0$-integers
of $F$ is defined as
$$o_{S_0}=\{ x\in F: \ \abs {x}_p  \leq 1 \ \text{for all} \
p <\infty_F \ \text{and} \ p \not\in S_0 \}. $$ Let
$F_p $ be the completion of $F$ at $p $ and $o_{F_p }$ be the local completion of $o_F$ at $p $
for each $p \in \Omega_F$. Write $o_{F_p }={F_{p }}$ for $p \in \infty_F$. We also denote the adele (resp. the
idele) of $F$ by $\Bbb A_F$ (resp. $\Bbb I_F$) and
$$F_{\infty}=\prod_{p \in \infty_F}F_p  .$$

A group $G$ of multiplicative type over $F$ is defined as a closed
subgroup of a torus over $F$. Let $\hat{G}=Hom_{\bar F}(G, \Bbb
G_m)$ be the character of $G$. Then the functor $G\mapsto \hat{G}$
is an anti-equivalence between the category of $F$-groups of
multiplicative type with the category of finitely generated abelian
groups with the continuous action of $Gal(\bar{F}/F)$ (see
\cite{Sko}). Moreover, $G$ is a torus if and only if $\hat{G}$ is
free. For any positive integer $k$, the finite group scheme $G[k]$
over $F$ stands for the kernel of the multiplication by $k$ over
$G$. We also use $\bar{F}$ to denote the algebraic closure of $F$.

Let $\bf X$ be a separated $o_F$-scheme of finite type whose
generic fiber $X_F$ is a trivial torsor of $G$. The obvious
necessary condition for ${\bf X}(o_F)\neq \emptyset$ is
\begin{equation}\label{loc} \prod_{p \in \Omega_F} {\bf
X}(o_{F_p })\neq \emptyset \end{equation} which is
assumed throughout this paper. The Brauer group $Br(X_F)$ of $X_F$
is defined as $$Br(X_F)=H^2_{et}(X_F, \Bbb G_m) \ \ \ \text{and} \ \
\ Br_1(X_F)=ker[Br(X_F)\rightarrow Br(\bar{X})]$$ where
$\bar{X}=X_F\times_F \bar{F}$. Since the image of $Br(F)$ induced by
the structure morphism lies in $Br_1(X_F)$, one defines
$$Br_a(X_F)=coker[Br(F)\rightarrow Br_1(X)]. $$
For any subgroup $\frak s$ of $Br_a (X_F)$, one can define the
integral Brauer-Manin set with respect to $\frak s$ as (see
\cite{CTX})
$$(\prod_{p \in \Omega_F}{\bf X}(o_{F_p }))^{\frak
s}=\{ (x_p )\in \prod_{p \in \Omega_F}{\bf X}(o_{F_p }): \ \ \sum_{p \in \Omega_F} inv_p (s(x_p ))=0, \ \ \ \forall s\in \frak s\}.$$

\bigskip

\section{Construction of {\bf X}-admissible groups}   \label{sec.construction}

In this section, we'll extend the construction of ${\bf
X}$-admissible groups in \cite{WX} for multi-norm tori to the
general groups of multiplicative type. By using the anti-equivalence
between the category of $F$-tori with the category of free $\Bbb
Z$-modules of finite rank with the continuous action of
$Gal(\bar{F}/F)$, one has the following lemma which should be standard.

\begin{lem} \label{inverse} Let $T_1$ and $T_2$ be two tori over $F$ and
$\phi: T_1\rightarrow T_2$ be a surjective morphism of tori. Then
there is a morphism of tori $$\psi: T_2 \rightarrow T_1 \ \ \
\text{such that} \ \ \ \phi\circ \psi =[l]$$ for some positive
integer $l$.
\end{lem}
\begin{proof} Since the category of $F$-tori is anti-equivalent to
the category of free $\Bbb Z$-modules of finite rank with the
continuous action of $Gal(\bar{F}/F)$, one has the injective
$Gal(\bar{F}/F)$-module homomorphism $\hat{\phi}:
\hat{T}_2\rightarrow \hat{T}_1$. Extending this map over $\Bbb Q$,
one obtains the injective $Gal(\bar{F}/F)$-module homomorphism
$$\hat{\phi}_\Bbb Q: \hat{T}_2\otimes_\Bbb Z \Bbb Q \longrightarrow
\hat{T}_1\otimes_\Bbb Z \Bbb Q .$$ Since the representation of
compact group over a field of characteristic 0 is semi-simple
(Maschke's Theorem), there is a $Gal(\bar{F}/F)$-module homomorphism
$$\hat{\psi'}_\Bbb Q: \hat{T}_1\otimes_\Bbb Z \Bbb Q \longrightarrow
\hat{T}_2\otimes_\Bbb Z \Bbb Q$$ such that $\hat{\psi'}_\Bbb Q \circ
\hat{\phi}_\Bbb Q =1$. Since $\hat{T}_1$ is of finite rank, there is
a positive integer $l$ such that $l\hat{\psi'}_\Bbb
Q(\hat{T}_1)\subseteq \hat{T}_2$. Then the homomorphism from $T_2$
to $T_1$ associated to $l\hat{\psi'}_\Bbb Q$ is as required.
\end{proof}

Fix a finite Galois extension $E/F$ such that $Gal(\bar{F}/E)$ acts
on $\hat{G}$ trivially. Let $M$ be a free $\Bbb Z$-module of finite
rank such that $f:  M\rightarrow \hat{G}$ is a surjective
homomorphism. One can extend this map to the $Gal(\bar{F}/F)$-module
surjective homomorphism $$ \Bbb Z[Gal(E/F)]\otimes_\Bbb Z
M\rightarrow \hat{G}, \ \ \ \sigma\otimes m \mapsto \sigma(f(m))
$$ where $Gal(\bar{F}/F)$ acts on $\Bbb Z[Gal(E/F)]\otimes_\Bbb Z
M$ as follows $$\tau \circ (\sigma \otimes m)=\bar{\tau}\sigma
\otimes m $$ for $\tau\in Gal(\bar{F}/F)$ and $\sigma\otimes m\in
\Bbb Z[Gal(E/F)]\otimes_\Bbb Z M$. By the anti-equivalence between
the category of multiplicative groups over $F$ with the category of
finitely generated $\Bbb Z$-modules with the continuous action of
$Gal(\bar{F}/F)$, there is an $F$-torus $T$ such that
\begin{equation}\label{resolution} 0 \rightarrow G\xrightarrow{\lambda}
\prod_{1}^n R_{E/F}(\Bbb G_m) \xrightarrow{\nu} T \rightarrow 0
\end{equation} where $n=rank_\Bbb Z (M)$.

Applying Lemma \ref{inverse} to (\ref{resolution}), one can fix a
homomorphism $$\mu: T \rightarrow \prod_{1}^n R_{E/F}(\Bbb G_m)$$
such that $\nu\circ \mu =[l]$ for some positive integer $l$.

\medskip

Since all morphisms are of finite type, we fix a finite subset $S_0\subset \Omega_F\setminus \infty_F$ such that

1) the groups $G$, $\prod_{1}^n R_{E/F}(\Bbb G_m)$ and $T$ over $F$
extend to the group schemes $\bold G$, $\bold R_{E/F}(\Bbb G_m)$ and
$\bold T$ over $o_{S_0}$ such that (\ref{resolution}) extends
to
$$0 \longrightarrow \bold G\longrightarrow \prod_{1}^n \bold R_{E/F}(\Bbb G_m) \longrightarrow
\bold T \longrightarrow 0 $$ over $o_{S_0}$ and $$\bold
R_{E/F}(\Bbb G_m)(o_{F_p })=o_{E_p }^\times
$$ for $p \not\in S_0$.

2) the trivial torsor $X_F$ of $G$ over $F$ extend to the trivial
torsor $\bold X_{S_0}=\bold X\times_{o_F} o_{S_0}$ of
$\bold G$ over $o_{S_0}$.

3) the homomorphism $\mu$ also extends to ${\bold T} \rightarrow
\prod_{1}^n \bold R_{E/F}(\Bbb G_m)$ over $o_{S_0}$.

4) any $l$-torsion point of $T(F_p )$ is contained in $\bold
T(o_{F_p })$ for all $p \not \in S_0$.

\medskip

For the above torus $T$,  the class group of $T$ with respect to
$S_0$ can be defined as follows
$$cls(T,S_0)=T(\Bbb A_F)/[T(F)+
(\prod_{p  \in S_0} T(F_p ) \times \prod_{p \not\in S_0} \bold T(o_{F_p }))]$$
and this group $cls(T,S_0)$ is finite by Theorem 5.1 in \cite{PR94}.
Denote $h(T,S_0)=\sharp cls(T,S_0)$.

 Let $A$ be a finite group scheme over $F$. One defines
$$ \cyr{X}_F^1(A)=ker[H^1(F,A)\longrightarrow \prod_{p <\infty_F} H^1(F_p , A)].$$

\begin{lem} \label{sha} For any torus $T$ over $F$, there is a positive integer $c=c(F,T)$ such that
$\cyr{X}_F^1(T[k])$ is killed by $c$ for all positive integer $k$.
\end{lem}
\begin{proof} Let $E$ be a finite Galois extension over $F$ such that $$T_E:=T\times_F E\cong \prod_{1}^n \Bbb
G_{m,E}$$ for some positive integer $n$. Therefore
$T_E[k]=\prod_{1}^n \mu_k$. Using the inflation-restriction exact sequence of group cohomology,
together with its functoriality properties, one obtains the following commutative diagram
$$ \begin{array}{ccccccc}
   H^1(E/F, T(E)[k]) & \longrightarrow  & H^1(F, T[k])  & \longrightarrow  & H^1(E, T_E[k])  \\
 & & \downarrow & & \downarrow  \\
 & & \prod_{p <\infty_F} H^1(F_p , T[k]) & \longrightarrow  & \prod_{P< \infty_E} H^1(E_P,T_E[k]).
 \end{array} $$
By (9.1.3) Theorem in \cite{NSW}, one knows that
$\cyr{X}_E^1(T_E[k])$ is killed by 2.  If $\mu(E)$ is the order of
the set of all roots of unity contained in $E$, then $H^1(E/F,
 T(E)[k])$ is killed by the maximal common divisor $\mu(E)$ and $[E:F]$. One can take
 $c=2(\mu(E),[E:F])$.
\end{proof}

From the above proof, one can further obtain a uniform bound of
$\cyr{X}_F^1(T[k])$ for all positive integer $k$. The above result
is good enough for our application.

\begin{lem} \label{split} If $T\cong \prod_{1}^d \Bbb G_m$ for some positive
integer $d$, then $$T(F)\cap [\prod_{p \in \infty_F}T(F_p )\times \prod_{p \not\in \infty_F} [2k]T(F_p )] \subseteq
[k] T(F) $$ for any positive integer $k$.
\end{lem}
\begin{proof} Any element  $$x\in T(F)
\cap [ \prod_{p \in \infty_F}T(F_p ) \times \prod_{p \not \in \infty_F} [2k]T(F_p  ) ]$$ can be written as
$x=(x_1,\cdots,x_d)$ where $x_i\in F^\times$ and $x_i\in (F_p ^\times)^{2k}$ for all $p < \infty_F$ with $1\leq i\leq d$.
Applying (9.1.3) Theorem in \cite{NSW} for $\mu_{2k}$, one obtains
that $x_i^2\in (F^\times)^{2k}$ for $1\leq i\leq d$. There is
$y_i\in F^\times$ such that $x_i=y_i^k$ or $x_i=-y_i^k$ with $1\leq
i\leq d$.

Suppose $x_i=-y_i^k$ for some $1\leq i\leq d$. Let $k=2^s k_1$ with
$2\nmid k_1$ and $\zeta_{2^{s+1}}$ be a primitive $2^{s+1}$-th roots
of unity. Then $\zeta_{2^{s+1}}\in F_p ^\times$ for all $p <\infty_F$. This implies that all primes $p <\infty_F$ are split completely in the abelian extension $F(\zeta_{2^{s+1}})/F$. By the Chebotarev density theorem, one concludes that
$\zeta_{2^{s+1}}\in F$. Therefore $-1=(\zeta_{2^{s+1}})^k \in
(F^\times)^k$ and the proof is complete.
\end{proof}

For a general torus $T$, one can fix the splitting field $K/F$ of
$T$ such that $$T_K=T\times_F K\cong \prod_{1}^d \Bbb G_{m,K}$$ for
some positive integer $d$. Since $K/F$ is a finite extension, the
$K$-rational torsion points of $T_K$ is finite. Fix a positive
integer $t$ such that $t$ kills all $K$-rational torsion points of
$T_K$.

\begin{lem} \label{power} Let $T$ be a torus over $F$ as above and $k$ be a positive integer. Then there is
a finite subset $S_1$ of $\Omega_F\setminus \infty_F$ containing
$S_0$ such that
$$T(F)\cap [\prod_{p \in S}[c\cdot 2t \cdot h(T,S_0)\cdot k] T(F_p ) \times \prod_{p \not\in S}\bold T(o_{F_p })]
 \subseteq [k]T(F) $$ for any finite subset $S\supseteq S_1 $, where $c=c(F,T)$ is as in Lemma \ref{sha} and $t$ is as above.
\end{lem}
\begin{proof} Let $U_0$ be a finite subset of $\Omega_K\setminus \infty_K$ such
that the ring $o_{U_0}$ of $U_0$-integers of $K$ is the
integral closure of $o_{S_0}$ inside $K$. Let $U$ be a finite
subset of $\Omega_K\setminus \infty_K$ containing $U_0$ such that
$$\bold T_{U_0}(o_U)=\prod_{1}^n o_U^\times $$ where
$\bold T_{U_0}=\bold T\times_{o_{S_0}} o_{U_0}$. By the
Dirichlet unit theorem, one obtains that $\bold T_{U_0}(o_U)$
is a finitely generated abelian group. Since $\bold T(o_{U_0})=\bold T_{U_0}(o_{U_0})$ is a subgroup of $\bold
T_{U_0}(o_U)$ and $\bold T(o_{S_0})$ is a subgroup of
$\bold T(o_{U_0})$, one concludes that $\bold T(o_{S_0})$ is a finitely generated abelian group. Therefore $\bold
T(o_{S_0})/[c\cdot 2t \cdot k]\bold T(o_{S_0})$ is
finite.

If the coset $\alpha + [c\cdot 2t\cdot  k]\bold T(o_{S_0})$ of
$\bold T(o_{S_0})/[c\cdot 2t \cdot k]\bold T(o_{S_0})$
satisfies that $$\alpha \not\in [c\cdot 2t \cdot k] T(F_{p })$$
for some prime $p  <\infty_F$, we fix such a prime $p _\alpha$. Let $\frak S$ be the set consisting of all such $p _\alpha$ and $S_1=S_0\cup \frak S$.

For any finite set $S\supseteq S_1$ and
$$ \alpha \in T(F)\cap [\prod_{p \in S}[c \cdot 2t \cdot h(T,S_0)\cdot k] T(F_p )
\times \prod_{p \not\in S}\bold T(o_{F_p })] , $$
there is $\beta_p \in T(F_p )$ such that $\alpha=[c\cdot
2t \cdot h(T,S_0)\cdot k] \beta_p $ for each $p \in S$.
Since $\beta_p $ can be viewed as an element in $T(\Bbb A_F)$
whose $p $ component is $\beta_p $ and the others are 1,
there is $\varpi_p \in T(F)$ such that $$\varpi_{p } +
[h(T,S_0)] \beta_p  \in [\prod_{p  \in S_0} T(F_p )
\times \prod_{p \not\in S_0} \bold T(o_{F_p })]$$
for each finite $p \in S\setminus S_0$. This implies that
$$\gamma =\alpha + [c\cdot 2t \cdot k](\sum_{p \in S\setminus (S_0\cup \infty_F)} \varpi_p ) \in
\bold T (o_{S_0}) .$$ Therefore the coset of $\gamma + [c\cdot
2t \cdot k] \bold T(o_{S_0})$ of $\bold T(o_{S_0})/[c\cdot 2t \cdot k]\bold T(o_{S_0})$ is not one of
the above mentioned cosets. One concludes that
$$ \alpha \in  T(F)\cap [ \prod_{p \in \infty_F}T(F_p )\times \prod_{p \not \in \infty_F} [c\cdot 2t \cdot
k]T(F_p )] .$$ Hence $$ \aligned & T(F)\cap [\prod_{p \in
S}[c \cdot 2t \cdot h(T,S_0)\cdot k] T(F_p )
\times \prod_{p \not\in S}\bold T(o_{F_p })] \\
\subseteq \ \ \ & T(F)\cap [ \prod_{p \in \infty_F}T(F_p )\times \prod_{p \not \in \infty_F} [c\cdot 2t \cdot
k]T(F_p )].
\endaligned $$

For any element $$x \in T(F)\cap [ \prod_{p \in
\infty_F}T(F_p )\times \prod_{p \not \in \infty_F} [c\cdot
2t \cdot k]T(F_p )],$$ there is $y\in T(F)$ such that
$[c]x=[c\cdot 2t\cdot k]y$ by Lemma \ref{sha}. Let $z=x-[2t\cdot k]
y$. Then $z\in T(F)[c]$. By applying Lemma \ref{split} to
$T_K=T\times_F K$ over $K$, one obtains that there is $w\in T(K)$
such that $z=[t\cdot k]w$. Since $z$ is a torsion point of $T$, one
has $w$ is also a torsion point of $T_K$. Since $t$ kills all
$K$-torsion points of $T_K$, one concludes that $z=0$. The proof is
complete.
\end{proof}

\begin{lem}\label{open} If $C_p $ is an open subgroup of $G(F_p )$ and $k$ is a positive integer,
then $$\lambda (C_p ) \cdot \mu ([k]T(F_p ))$$ is an open
subgroup of $$\prod_{1}^n R_{E/F}(\Bbb G_m)(F_p )=\prod_{1}^n
E_p ^\times$$ for $p $-adic topology with $p <\infty_F$.
\end{lem}
\begin{proof} Since $C_p $ is an open subgroup of $G(F_p )$ and $\lambda[G(F_p )]$ is a closed subgroup of $$\prod_1^n
R_{E/F}(\Bbb G_m)(F_p )=\prod_1^n E_p ^\times , $$ there
is a positive integer $a$ such that
$$[\prod_1^n(1+p ^a o_{E_{p }})]\cap \lambda[G(F_p )]
\subseteq \lambda(C_p ).
$$ Since $\mu\circ \nu$ is a continuous endomorphism of $$\prod_1^n
R_{E/F}(\Bbb G_m)(F_p )=\prod_1^n E_p ^\times, $$ there is
a positive integer $a_1>a$ such that
$$\mu
\circ \nu [\prod_{1}^n (1+p ^{a_1} o_{E_{p }})]\subseteq \prod_{1}^n(1+p ^a o_{E_{p }}).$$
 By Hensel's lemma,
there is a positive integer $a_2>a_1$ such that
$$(1+p ^{a_2} o_{E_{p }})\subseteq (1+p ^{a_1} o_{E_{p }})^{lk} .$$

For any $x \in \prod_{1}^n (1+p ^{a_2} o_{E_{p }})
$, there is $$y\in \prod_{1}^n (1+p ^{a_1} o_{E_{p }}) \ \ \ \text{such that} \ \ \ x=y^{lk} .$$ Then
$$x=y^{lk}=[y^{lk}(\mu\circ \nu (y^k))^{-1}](\mu\circ \nu (y^k)).
$$ Since
$$\nu [y^{lk}(\mu\circ \nu
(y^k))^{-1}]=[l]\nu(y^k)-\nu\circ\mu\circ\nu(y^k)=0 ,$$ one
concludes $$y^{lk}(\mu\circ \nu (y^k))^{-1} \in \lambda(C_p ).
$$ Therefore
$$\prod_{1}^n (1+p ^{a_2} o_{E_{p }}) \subseteq
\lambda (C_p ) \cdot \mu ([k]T(F_p ))$$ and the proof is
complete. \end{proof}

Since $\bf X$ is separated over $o_{F}$, one can view ${\bf
X}(o_{F_p })$ as an open subset of $X_F(F_p )$ by
the natural map for any $p  \in\Omega_F$.

\begin{definition}\label{stlocal} Define $$Stab({\bf X}(o_{F_p }))=\{ g\in
G(F_p ):  \ \ g{\bf X}(o_{F_p })={\bf X}(o_{F_p })\}$$ for $p  \in \Omega_F$.

\end{definition}

It is clear that $Stab({\bf X}(o_{F_p }))$ is an open
subgroup of $G(F_p )$. Since $X_F$ is a trivial torsor of $G$
over $F$, one has that $$Stab({\bf X}(o_{F_p }))={\bf
G}(o_{F_p })$$ for almost all $p $ in $\Omega_F$.

\begin{definition}\label{stadelic} Define $$Stab_{\Bbb A}({\bf X})=
\prod_{p \in \Omega_F}Stab({\bf X}(o_{F_p })). $$
\end{definition}

It is clear that $Stab_{\Bbb A}({\bf X})$ is an open subgroup of $
G(\Bbb A_F)$. The map given by (\ref{resolution})  $$\lambda:
G\longrightarrow \prod_{1}^n R_{E/F}(\Bbb G_m)$$ induces the
homomorphism
$$\lambda: \ \ G(\Bbb A_F) \longrightarrow \prod_{1}^n \Bbb I_E . $$

\begin{definition}\label{adm} An open subgroup $\Xi$ of $\prod_{1}^n \Bbb
I_{E}$ is called $\bf X$-admissible if
$$\lambda [Stab_\A({\bf X})]\subseteq \Xi$$ and the
induced map
$$ \lambda: \ \ \ G(\Bbb A_F)/G(F)Stab_\A({\bf X}) \longrightarrow
\prod_{1}^n \Bbb I_{E}/(\prod_{1}^n E^\times)\cdot \Xi $$ is
injective. \end{definition}

The main result of this section is to show the existence of the
admissible subgroups.

\begin{thm}\label{exsitence} If $\bf X$ is a separated scheme over $o_F$ of finite type
such that the generic fiber $X_F$ is a trivial torsor of a
multiplicative group $G$ satisfying (\ref{resolution}), then the
$\bf X$-admissible subgroups of $\prod_{1}^n \Bbb I_E$ always exist.
\end{thm}
\begin{proof}
By the condition 1) and 2) for the choice of $S_0$, one has
$$\lambda [Stab({\bf X}(o_{F_p }))]=\lambda [{\bf
G}(o_{F_p })]=ker[\nu: (\prod_{1}^no_{E_p }^\times) {\longrightarrow} {\bf T}(o_{F_p })]$$ for all
$p \not \in S_0$. For each $p \in S_0$, one can fix a
positive integer $r_p $ such that $r_p $ kills all torsion
points of $T(F_p )$. Let $$r=\prod_{p \in S_0} r_p
.$$

Let $S_1$ be a finite subset of $\Omega_F$ outside $\infty_F$  such
that Lemma \ref{power} holds for $k=l\cdot r$. Define
$$\Xi= [\prod_{p \in S_1}\lambda (Stab({\bf X}(o_{F_p })))\cdot
\mu ([c\cdot 2t \cdot h(T,S_0) \cdot r ]T(F_p ))]\times (
\prod_{p \not \in S_1}\prod_{1}^n o_{E_p }^\times).
$$  By Lemma \ref{open},
one has that $\Xi$ is an open subgroup of $\prod_{1}^n \Bbb I_{E}$
and $$\lambda [Stab_\A({\bf X})]\subseteq \Xi .$$

Consider $\sigma \in G(\A_F)$ such that $\lambda(\sigma) =a \cdot i$
with $$a \in \prod_{1}^n E^\times \ \ \ \text{and} \ \ \ i\in \Xi .
$$ Then
$$ \nu (a)\in [\prod_{p \in S_1}[c\cdot 2t \cdot h(T,S_0) \cdot l\cdot r ]T(F_p ) \times
\prod_{p \not\in S_1}{\bf T}(o_{F_p })].$$ By Lemma
\ref{power}, there is $u\in T(F)$ such that $\nu (a)=[l\cdot r](u)$.
Since $$\nu (a\cdot(\mu ([r](u)))^{-1})=\nu (a)-[l\cdot r](u)=0 , $$
one obtains that
$$a\cdot(\mu ([r](u)))^{-1} \in \lambda(G(F)) . $$

Let $i=( i_p  )_{p \in \Omega_F}$. One can write $$i_p =s_p  \cdot \mu([c\cdot 2t \cdot h(T,S_0) \cdot r ](n_p ))$$ with
$$s_p \in \lambda(Stab({\bf X}(o_{F_p }))) \ \ \ \text{and} \ \ \
n_p \in T(F_p )$$ for $p \in S_1$. This implies that
$$[l\cdot r](u + [c\cdot 2t \cdot h(T,S_0)](n_p ))=0$$ for $p \in S_1$. Therefore $u + [c\cdot 2t \cdot h(T,S_0)](n_p )$ is
a torsion point of $T(F_p )$ for each $p \in S_1$.

If $p \in S_0$, then $r$ kills $u + [c\cdot 2t \cdot
h(T,S_0)](n_p )$. One concludes that
$$\mu([r](u))\cdot\mu([c\cdot 2t \cdot h(T,S_0)\cdot r](n_p ))=1.
$$

If $p  \in S_1\setminus S_0$, one has $$ [r]u + [c\cdot 2t
\cdot h(T,S_0)\cdot r](n_p ) \in \bold T(o_{F_p })$$
by the condition 4) for the choice of $S_0$. By the condition 1) and
3) for the choice of $S_0$, one obtains that $$\mu([r]u)\cdot
\mu([c\cdot 2t \cdot h(T,S_0)\cdot r](n_p ))\in \prod_{1}^n
o_{E_p }^\times .$$

Therefore one concludes that
$$\mu([r]u) \cdot i\in \lambda [Stab_\A({\bf X})] \ \ \ \text{and} \ \ \ \sigma
\in G(F)Stab_\A({\bf X}) .$$  The proof is complete.
\end{proof}

Since $X_F$ is a trivial $G$-torsor over $F$, one can fix a rational
point $Q\in X_F(F)$ which induces the isomorphism $X_F\cong G$ as
$F$-varieties. Since $\bf X$ is separated over $o_F$, the
natural morphism
$$\prod_{p \in \Omega_F} \bold X(o_{F_p })
\longrightarrow X_F(\Bbb A_F)
$$ is injective. Define
$$f: \ \ \prod_{p \in \Omega_F} \bold X(o_{F_p }) \rightarrow X_{F}(\A_F)
\cong G(\A_F) \xrightarrow \lambda \prod_{1}^n \Bbb I_{E}. $$

\begin{cor} \label{main} Let $\Xi$ be an $\bf X$-admissible subgroup of
$\prod_{1}^n \Bbb I_{E}$. Then
$$\bold X(o_F) \neq \emptyset \ \ \ \ \ \text{if and only if} \ \ \ \ \
f[\prod_{p \in \Omega_F} \bold X(o_{F_p })]\cap
[(\prod_{1}^n E^\times )\cdot \Xi] \neq \emptyset.$$
\end{cor}
\begin{proof} Since $\bold X$ is separated over $o_F$, one has
$\bold X(o_F)\subseteq \bold X(o_{F_p })$ for all
$p \in \Omega_F$ and
$$\bold X(o_F) \subseteq \prod_{p \in \Omega_F} \bold X(o_{F_p })$$ by the diagonal map. If $\bold X(o_F)\neq \emptyset$,
then
$$f[\prod_{p \in \Omega_F} \bold X(o_{F_p })]\cap
[(\prod_1^n E^\times)\cdot \Xi] \supseteq f_E[\bold X(o_F)]\cap (\prod_1^n E^\times) \neq \emptyset$$ and the necessity
follows.

Conversely, there is  $$y_A \in \prod_{p \in \Omega_F} \bold
X(o_{F_{p }}) \ \ \ \text{ such that } \ \ \ f(y_A)\in
(\prod_1^n E^\times)\cdot \Xi .$$ By Definition \ref{adm}, there are
$\varrho \in G(F)$ and $\sigma_A\in Stab_A(\bold X)$ such that
$y_A=\varrho \sigma_A (Q)$. This implies that $$\varrho
(Q)=\sigma_A^{-1} (y_A)\in \prod_{p \in \Omega_F} \bold X(o_{F_p }).
$$ Therefore $\varrho (Q) \in \bold X(o_F)\neq \emptyset$  and
the proof is complete.
\end{proof}

If $\Xi$ is an $\bf X$-admissible subgroup of $\prod_{1}^n \Bbb
I_{E}$, there is an open subgroup $\Xi_i$ of $\Bbb I_{E}$ for each
$1\leq i\leq n$ such that $$\prod_{i=1}^n \Xi_i \subseteq \Xi .$$ By
the class field theory, there is a finite abelian extension
$K_{\Xi_i}/E$ such that the Artin map
\begin{equation}\label{art} \psi_{K_{\Xi_i}/E}: \ \ \ \Bbb
I_{E}/E^\times \Xi_i \cong Gal(K_{\Xi_i}/E)
\end{equation} gives the isomorphism for $1\leq i\leq n$. Projecting
the image of $f$ to the $i$-th component, one can define
$$f_{i}: \ \ \prod_{p \in \Omega_F} \bold X(o_{F_p }) \longrightarrow \prod_{1}^n \Bbb I_{E}
\longrightarrow \Bbb I_{E} $$ for $1\leq i\leq n$.

\begin{cor}\label{ef} With the notation as above,
$\bold X(o_F)\neq \emptyset$ if and only if there is
$$x_A\in\prod_{p \in \Omega_F} \bold X(o_{F_p }) \ \ \ \text{such that} \ \ \
\psi_{K_{\Xi_i}/E}(f_{i}(x_A))=1 $$ in $Gal(K_{\Xi_i}/E)$ for all
$1\leq i\leq n$.
\end{cor}
\begin{proof} Since $\psi_{K_{\Xi_i}/E}(f_i(x_A))=1$ for $1\leq i\leq n$, one
has $$f(x_A)\in [\prod_{i=1}^n (E^\times \cdot \Xi_i)]\subseteq
(\prod_{1}^n E^\times)\cdot \Xi .$$ Then the result follows from the
same argument as those in Corollary \ref{main}.
 \end{proof}

\bigskip

\section{Brauer-Manin obstruction}   \label{sec.brauer}

In this section, we will interpret $\bold X$-admissible subgroups in
terms of Brauer-Manin obstruction. This translation has been done
for one dimension tori in \cite{WX}. For general case, the argument
is similar. We keep the same notation as the previous section.

Since $X_F$ is a trivial torsor of $G$ over $F$, the fixed rational
$F$-point $Q$ gives $Gal(\bar{F}/F)$-module homomorphism $ \hat{G}
\longrightarrow  \bar F[X]^\times $, where $\bar F[X]^\times$ is the
global sections of $\bar X$. Then one obtains the homomorphism
\begin{equation} \label{shapiro}  H^2(E, M)\cong H^2(F,\Bbb
Z[Gal(E/F)]\otimes_\Bbb Z M) \rightarrow H^2(F, \hat{G}) \rightarrow
H^2(F,\bar F[X]^\times)  \end{equation} by Shapiro's lemma. Applying
Hochschild-Serre spectral sequence (see Theorem 2.20 of Chapter III
in \cite{Milne80}) in \'etale cohomology
$$H^p(F,H^q(\bar X, \Bbb G_m))\Rightarrow H^{p+q}(X_F,\Bbb G_m),$$
one has
$$ \phi_Q : \ \ \  H^2(E, M) \longrightarrow H^2(F,\bar
F[X]^\times)\longrightarrow Br_1(X_F). $$

Fix the basis $\{e_1, \cdots, e_n\}$ of $M$ such that the projection
$$ p_i: \ \ \ \prod_{1}^n R_{E/F}(\Bbb G_m) \longrightarrow
R_{E/F}(\Bbb G_m)$$ to the $i$-th component is given by $\Bbb Z
e_i\subseteq M$ for each $1\leq i\leq n$. Then the evaluation of the
following morphism by using the fixed rational point $Q$
\begin{equation} \label{evaluation} X_F \xrightarrow{Q} G \xrightarrow{\lambda} \prod_{1}^n
R_{E/F}(\Bbb G_m) \xrightarrow{p_i} R_{E/F}(\Bbb G_m) \end{equation}
at $\prod_{p \in \Omega_F} \bold X(o_{F_p })$ is
$f_i$ in \S \ref{sec.construction}.

Let $\Xi$ be an $\bf X$-admissible subgroup of $\prod_{1}^n \Bbb
I_{E}$ and $\Xi_i$ be an open subgroup of $\Bbb I_{E}$ such that
$$\prod_{i=1}^n \Xi_i \subseteq \Xi$$ for each $1\leq i\leq n$. Let
$K_{\Xi_i}/E$ be a finite abelian extension satisfying (\ref{art})
for $1\leq i\leq n$. Since
$$H^2(Gal(K_{\Xi_i}/E), \Bbb Z e_i) \cong H^1(Gal(K_{\Xi_i}/E), \Bbb
Q/\Bbb Z) = Hom(Gal(K_{\Xi_i}/E), \Bbb Q/\Bbb Z)$$ is finite, one
has the image of the inflation map composed with $\phi_Q$
$$H^2(Gal(K_{\Xi_i}/E), \Bbb Z e_i) \rightarrow H^2(E,\Bbb Z
e_i)\subseteq H^2(E,M) \xrightarrow{\phi_Q} Br_1(X_F) $$ is finite
and denoted by $b(\Xi_i)$.

\begin{definition} For any $\bf
X$-admissible subgroup  $\Xi$ of $\prod_{1}^n \Bbb I_{E}$ and an
open subgroup $\Xi_i$ of $\Bbb I_{E}$ with
$$\prod_{i=1}^n \Xi_i \subseteq \Xi$$ for each $1\leq i\leq n$, one defines
the finite group $b(\Xi)$ of $Br_1(X_F)$ to be generated by
$b(\Xi_i)$ for $1\leq i\leq n$.
\end{definition}

One can reformulate Corollary \ref{main}  in terms of Brauer-Manin
obstruction.

\begin{thm}\label{brauer} Let $\Xi$ be an $\bf
X$-admissible subgroup of $\prod_{1}^n \Bbb I_{E}$. Then
$$\bold X(o_F) \neq \emptyset \ \ \ \text{ if and only if} \ \
\ \ [\prod_{p \in \Omega_F} \bold X(o_{F_p })]^{b(\Xi)}\neq \emptyset.$$
\end{thm}
\begin{proof} The necessity follows from the class field theory and one only needs to show the sufficiency.
Applying Galois cohomology to the short exact sequence
$$0\longrightarrow \Bbb Z\longrightarrow \Bbb Q \longrightarrow \Bbb
Q/\Bbb Z \longrightarrow 0 $$ with the trivial action, one obtains
 \begin{equation} \label{connection} \delta_i \ \  Hom(Gal(K_{\Xi_i}/E), \Bbb Q/\Bbb Z)
\cong H^2(Gal(K_{\Xi_i}/E), \Bbb Z)  \end{equation} with $1\leq
i\leq n$.  For any $\chi \in Hom(Gal(K_{\Xi_i}/E), \Bbb Q/\Bbb Z)$,
the cup product gives
$$\xi_i=e_i\cup \delta_i(\chi) \in H^2(Gal(K_{\Xi_i}/E), {\Bbb
Z}e_i) \ \ \ \text{and} \ \ \ \beta_i =\phi_Q (\xi_i)\in b(\Xi_i)
$$ with $1\leq i\leq n$.

Let $$(x_p )_{p \in \Omega_F}\in [\prod_{p \in
\Omega_F} \bold X(o_{F_p })]^{b(\Xi)}  $$ and evaluate
$\beta_i$ at $(x_p )_{p \in \Omega_F}$. Then
$$inv_p (\beta_i (x_p ))=\sum_{P |p } inv_{P }(f_i(x_p )\cup \delta_i(\chi))$$ by (\ref{shapiro}),
(\ref{evaluation}), (8.1.4) Proposition and (7.1.4) Corollary in
\cite{NSW}, where $P $'s are all primes in $E$ above $p $.
One has
$$\sum_{P}inv_P (f_i(x_p )\cup
\delta_i(\chi))=0 $$ where $P$ runs over all primes of $E$.
This implies
$$\chi(\psi_{K_{\Xi_i}/E} (f_i[(x_p )_{p \in \Omega_F}]))=0$$
for all $\chi \in Hom(Gal(K_{\Xi_i}/E), \Bbb Q/\Bbb Z)$ by (8.1.11)
Proposition in \cite{NSW} and (\ref{art}). Therefore
$$\psi_{K_{\Xi_i}/E} (f_i[(x_p )_{p \in \Omega_F}])=0 $$ for $1\leq i\leq n$.
The result follows from Corollary \ref{ef}.
\end{proof}

\bigskip

\section{Strong approximation}   \label{sec.app}

As an application, we'll prove the strong approximation theorem for
groups of multiplicative type. Such strong approximation theorem has
been established for algebraic tori by Harari in \cite{Ha08} and
generalized to connected reductive groups by Demarche in \cite{Dem}.
We give the different proof for groups of multiplicative type where
$G$ is not assumed to be connected. We keep the same notation as
that in the previous sections.

By (8.1.16) Corollary in \cite{NSW}, the evaluation gives the pair
\begin{equation} \label{pair} G(\Bbb A_F) \times H^2(F, \hat{G})
\xrightarrow{inv} \Bbb Q/\Bbb Z
\end{equation}
such that $G(F)$ is lying in the left kernel of the above pair. Let
$G(F_\infty)^0$ be the set of elements in $G(F_\infty)$ which are
lying in the left kernel of (\ref{pair}).

\begin{thm}\label{sat} $G(F)\cdot G(F_\infty)^0$ is dense in the left kernel of (\ref{pair}).
\end{thm}
\begin{proof} By the definition of $G(F_\infty)^0$,  one only needs to show that
$G(F) \cdot U$ contains the left kernel of (\ref{pair}) for any open
subgroup $U$ of $G(\Bbb A_F)$ which contains $G(F_\infty)$. Without
loss of generality, one can assume that
$$U= G(F_\infty)\times \prod_{p < \infty_F} U_p  $$ where $U_p $ is an open subgroup
of $G(F_p )$ for all $p < \infty_F$ such that there is a
finite subset $S_0$ of $\Omega_F$ outside $\infty_F$ satisfying
condition 1)-4) in \S \ref{sec.construction} and
$$U_p ={\bf G}(o_{F_p })$$ for $p \not\in S_0$.
For each $p \in S_0$, one can fix a positive integer $r_p $ such that $r_p $ kills all torsion points of $T(F_p )$.
Let $r=\prod_{p \in S_0} r_p $ and  $S_1$ be a finite
subset of $\Omega_F$ outside $\infty_F$  such that Lemma \ref{power}
holds for $k=l\cdot r$. Define
$$\Xi= [\prod_{p \in S_1}\lambda (U_p )\cdot
\mu ([c\cdot 2t \cdot h(T,S_0) \cdot r ]T(F_p ))]\times (
\prod_{p \not \in S_1}\prod_{1}^n o_{E_p }^\times).
$$  By Lemma \ref{open},
one has that $\Xi$ is an open subgroup of $\prod_{1}^n \Bbb I_{E}$
and $\lambda (U) \subseteq \Xi$. Moreover the induced map
\begin{equation} \label{embed}  \lambda: \ \ \ G(\Bbb A_F)/G(F)\cdot
U \longrightarrow \prod_{1}^n \Bbb I_{E}/(\prod_{1}^n E^\times)\cdot
\Xi \end{equation} is injective by the exact same arguments in
Theorem \ref{exsitence}.

Let $\Xi_i$ be an open subgroup of $\Bbb I_{E}$ for each $1\leq
i\leq n$ such that
$$\prod_{i=1}^n \Xi_i \subseteq \Xi $$ and $K_{\Xi_i}/E$ be a finite
abelian extension such that $(\ref{art})$ holds for $1\leq i\leq n$.
Then
$$0\rightarrow \prod_{i=1}^n
Hom(Gal(K_{\Xi_i}/E), \Bbb Q/\Bbb Z)\rightarrow \prod_1^n
Hom(Gal(\bar{F}/E), \Bbb Q/\Bbb Z)\cong H^2(E,M). $$

Let
$$g=(g_p ) \in G(\Bbb A_F) \ \ \ \text{ and } \ \ \ \lambda(g)=(\sigma_1, \cdots,
\sigma_n)$$ with $\sigma_i\in \Bbb I_E$ for $1\leq i\leq n$. By the
functoriality of evaluation and (8.1.11) Proposition in \cite{NSW},
one obtains
$$\sum_{i=1}^n \chi_i(\psi_{K_{\Xi_i}/E}(\sigma_i))= inv(g, \lambda^*(\chi_1,\cdots,\chi_n)) $$
by (\ref{pair}) for all
$$(\chi_1, \cdots, \chi_n)\in \prod_{i=1}^n Hom(Gal(K_{\Xi_i}/E),
\Bbb Q/\Bbb Z) . $$

If $g$ is in the left kernel of (\ref{pair}), then $\sigma_i\in
E^\times \Xi_i$ for $1\leq i\leq n$. Therefore
$$\lambda(g) \in (\prod_1^n E^\times )\cdot \Xi .
$$ The injectivity of (\ref{embed}) implies that $g\in G(F)\cdot U$.
The proof is complete.
\end{proof}

When $G$ is a finite commutative group scheme over $F$, the above
result follows from Poitou-Tate (see (8.6.13) in \cite{NSW}).

\begin{rem}\label{infty} When $F_{p }=\Bbb R$, one
defines
$$ N_G(F_{p })= \{ x+\bar x: x\in  G(\Bbb C) \ \text{and $\bar x$ is the conjugate point of $x$} \}.$$
When $F_{p }=\Bbb C$, one defines  $$N_G(F_{p })=G(F_{p }).$$ Let
$$N_G(F_\infty)=\prod_{p \in \infty_F} N_G(F_{p }). $$

By the local duality over $\Bbb R$ by (7.2.17) Theorem in
\cite{NSW}, one has that $$N_G(F_\infty)\subseteq G(F_\infty)^0 $$
with the finite index. One can expect that $G(F)\cdot N_G(F_\infty)$
is dense in the left kernel of (\ref{pair}) by more careful study on
archimedean places.
\end{rem}

\bigskip

\section{Sum of two squares }
The natural extension of Fermat-Gauss' theorem about the sum of two
squares over $\Z$ to the ring of integers of quadratic field $F$ was
already studied by Niven for $F=\Bbb Q(\sqrt{-1})$ in \cite{Ni} and
by Nagell for $F=\Bbb Q(\sqrt{d})$ where $$d=\pm 2, \pm 3, \pm 5,
\pm 7, \pm 11, \pm 13, \pm 19, \pm 43, \pm 67, \pm 163$$ in
\cite{Na53} and \cite{Na61}. Both followed Gauss' original idea.
Therefore the class number one is assumed and the results obtained
there always satisfy the local-global principle. In this section, we
will apply the method developed in the previous sections to study
this question and gives an example of which the local-global
principle is no longer true.

It should be pointed out that our method only produces the idelic
class groups of $\Bbb Q(\sqrt{d}, \sqrt{-1})$ for solving the
problem of sum of two squares. In order to get the explicit
conditions for the sum of two squares, one needs further to
construct the explicit abelian extensions of $\Bbb Q(\sqrt{d},
\sqrt{-1})$ corresponding to the idelic class groups by class field
theory. Such explicit construction is a wide open problem (Hilbert's
12-th problem) in general but ad hoc method. In his series papers
\cite{W1} and \cite{W2}, the first named author gives explicit
construction for infinitely many $d$ and solves the sum of two
squares over infinitely many quadratic fields.

Let $l$ be a prime with $l\equiv -1 \mod 8$ and $F=\Q(\sqrt{-2l})$.
We will study the sum of two squares over $o_F$. Let
$$E=F(\sqrt{-1}) \ \ \ \text{and} \ \ \ \Theta=E(\sqrt[4]{l}) . $$

\begin{lem} \label{unr}The field $\Theta/E$ is
unramified over all primes except the prime above $l$.
\end{lem}
\begin{proof} Since $2$ is totally ramified in $E/\Bbb Q$, there is a unique prime $v$ of $E$ over $2$.
One only needs to show that $v$ is unramified in $\Theta/E$.

Since $\sqrt{l}\not \in E$, one obtains that $x^4-l$ is irreducible
over $E$ by the Kummer theory. By Hensel's Lemma, there is $s\in
(1+4 \Z_2)$ such that $s^2=-l$ over $\Z_2$. Since
$$E_v=\Q_2(\sqrt{-1},\sqrt{2})=\Q_2 (\zeta_8) $$ where $\zeta_8$ is the primitive $8-th$ root of unity, one has
$$x^4-l=x^4+s^2=x^4-s^2(\sqrt{-1})^2=(x^2-s\sqrt{-1})(x^2+s\sqrt{-1})$$ over
$E_v$. By $\zeta_8^2=\sqrt{-1}$ and $-1=(\sqrt{-1})^2$, one
concludes that $\Theta/E$ is unramified over $v$.
\end{proof}

It is clear that $l$ is ramified in $F/\Q$ and the unique prime of
$F$ over $l$ is also denoted by $l$. Since $l\equiv 3 \mod 4$, one
has that $l$ is inert in $E/F$. The unique prime of $E$ above
$l$ is denoted by $l$ as well and
$$E_l=\Q_l(\sqrt{-1}, \sqrt{l}).$$

\begin{lem} \label{computation} If $N_{E_l/F_l}(\xi)=1$, then the
Hilbert symbols over $E_l$
$$(\xi,\sqrt{l})_l=(\xi,-\sqrt{l})_l=1 .$$

If $N_{E_l/F_l}(\xi)=-1$, then the Hilbert symbols over
$E_l$
$$(\xi,\sqrt{l})_l=(\xi,-\sqrt{l})_l=-1 .$$
\end{lem}
\begin{proof} If $N_{E_l/F_l}(\xi)=1$, then $\xi$ is a unit of $E_l$.
Therefore
$$(\xi,\pm \sqrt{l})_{l} = (\xi, \sqrt{-2l})_l \cdot (\xi, \pm \sqrt{-2})_l =
(N_{E_l/F_l}(\xi),\sqrt{-2l})_l =1 $$ by 63:11a in
\cite{OM} and (1.5.3) Proposition and (7.1.4) Corollary in
\cite{NSW}.

If $N_{E_l/F_l}(\xi)=-1$, then
$$(\xi,\pm \sqrt{l})_{l} = (\xi, \sqrt{-2l})_l \cdot (\xi, \pm \sqrt{-2})_l =
(N_{E_l/F_l}(\xi),\sqrt{-2l})_l =(-1,2l)_l=-1 $$
by 63:11a in \cite{OM} and (1.5.3) Proposition and (7.1.4) Corollary
in \cite{NSW}.
\end{proof}

Let $B=o_F + o_F \sqrt{-1}$ be the order of $E$ and
$K_B$ be the ring class field of $E$ defined by $B$.

\begin{prop}\label{extra} Let $\bold X$ be the scheme
defined by $x^2+y^2=n$ for some non-zero integer $n$ over $o_F$. Then $\bold X(o_F)\neq \emptyset$ if and only if there
is
$$\prod_{w\in \Omega_F}(x_w,y_w)\in \prod_{w\in \Omega_F} \bold X(
o_{F_w})$$ such that  $$\psi_{K_B/E}(\tilde f_E[\prod_{w\in
\Omega_F}(x_w,y_w)])=1 \ \ \ \text{and} \ \ \ \psi_{\Theta/E}(\tilde
f_E[\prod_{w\in \Omega_F}(x_w,y_w)])=1 $$ where $K_B$ and $\Theta$
are defined as above, both $\psi_{K_B/E}$ and $\psi_{\Theta/E}$ are
the Artin maps and
$$\tilde f_E[(x_w,y_w)]= \begin{cases} (x_w+y_w \sqrt{-1},
x_w-y_w\sqrt{-1}) \ \ \
& \text{if $w$ splits in $E/F$} \\
x_w+y_w\sqrt{-1} \ \ \ & \text{otherwise.}
\end{cases} $$
\end{prop}
\begin{proof} Define $E_w=E\otimes_F F_w$ for any $w\in \Omega_F$.
Let $B_w$ be the completion of $B$ inside $E_w$ for $w<\infty_F$ and
$B_w=E_w$ for $w\in \infty_F$. Let
$$SO_\A(B)=\{ \sigma\in R_{E/F}^1(\G_m)(\A_F): \ \sigma B= B\}. $$

By Lemma \ref{unr} and Lemma \ref{computation}, the natural group
homomorphism
$$\lambda_E :  R_{E/F}^1(\G_m)(\A_F)/R_{E/F}^1(\G_m)(F)SO_\A(B)\rightarrow [\Bbb I_E/(E^\times  N_{\Theta/E}(\Bbb
I_{\Theta}))] \times [\Bbb I_E/(E^\times \prod_{w\in \Omega_F}
B_w^\times)]$$ is well-defined.

If $u\in ker \lambda_E$, there are $$\alpha\in E^\times \ \ \ \text{
and } \ \ \ i\in \prod_{w\in \Omega_F} B_w^\times$$ with
$\lambda_E(u)=\alpha i$. Therefore
$$N_{E/F}(\alpha)=N_{E/F}(i)^{-1} \in F\cap
(\prod_{w\in \Omega_F}o_{F_w}^\times )=\{\pm 1\} . $$

Suppose $N_{E/F}(\alpha)=N_{E/F}(i)=-1$. Write $$i=(i_w)_{w}\in
\prod_{w\in \Omega_F} L_w^\times .$$ Since $\Theta/E$ is unramified
over all primes of $E$ except $l$ by Lemma \ref{unr}, one has
$\psi_{\Theta/E}(i_w)$ is trivial for all primes $w\neq l$,
where $i_w$ is regarded as an idele whose $w$-component is $i_w$ and
1 otherwise. Since $$N_{E/F}(i_l)=N_{E_l/F_l}(i_l)=-1 , $$ one gets
$$\psi_{\Theta/E}(\alpha i)=\psi_{\Theta/E}(i)=\psi_{\Theta/E}(i_l)=-1 \in \mu_4 $$
by Lemma \ref{computation}, where $\mu_4$ is the set of 4-th roots
of unity and $Gal(\Theta/E)\cong \mu_4$. This contradicts to $u \in
ker \lambda_E$.

Therefore $$N_{E/F}(\alpha)=N_{E/F}(i)=1 .$$ This implies that
$$\alpha\in R_{E/F}^1(\G_m)(F) \ \ \ \text{ and }  \ \ \  i\in SO_\A(B). $$ One concludes
that
 $\lambda_E$ is injective. The
result follows from Corollary \ref{ef}.
\end{proof}

Finally we will give an explicit example by Proposition \ref{extra}
such that the local-global principle is not true.

Let
$$N_{F/\Q}(\delta)=2^{s_1}7^{s_2}p_1^{e_1}\cdots p_g^{e_g} \ \ \ \text{ and
 } \ \ \ D(\delta)=\{p_1, \cdots, p_g \}$$
where $\delta=a+b\sqrt{-14}\ (a,b \in \Z)$ is an integer of
$F=\Q(\sqrt{-14})$. Denote $a=7^{s_3}a_1$ with $a_1\in \Z, 7\nmid a_1$ and
$$\aligned
& D_1=\{p\in D(\delta): (\frac{-1}{p})=(\frac{14}{p})=1 \text{ and } (\frac{7}{p})=-1\
\}\cr & D_2= \{p\in D(\delta): (\frac{-1}{p})=-(\frac{14}{p})=1
\text{ and } (\frac{7}{p})=-1\ \} \cr & D_3= \{p\in D(\delta):
(\frac{-1}{p})=(\frac{14}{p})=1\text{ and } (\frac{7}{p})_4=-1 \}.
\endaligned $$
It is clear that $e_i$ is even for $p_i\in D_2$.

\begin{exa} \label{pell} Let $F=\Q(\sqrt{-14})$ and $\delta\in o_F$ as above.
Then $\delta$ can be written as a sum of two squares over $o_F$ if and only if

(1) $\delta$ can be written a sum of two squares over $o_{F_w}$ at every place $w$ of $F$.

(2) $D_1\neq \emptyset$; or
$$(\frac{a_1}{7})=(-1)^{s_1+s_2/2+\sum_{p_i\in D_2}\frac{e_i}{2}+\sum_{p_i\in
D_3}e_i}$$ for $D_1=\emptyset$.
\end{exa}

\begin{proof} Let $H$ be the Hilbert class field of $F$. By some computations, we know $K_B=H_E=H.E$ (with degree 4), where $H_E$ is the Hilbert class field of $E$. For any $\prod_{w\in \Omega_F}(x_w,y_w)\in \prod_{w\in \Omega_F}\bold X(o_{F_w})$, we have
$$\aligned &\psi_{K_B/E}(\tilde f_E[\prod_{w\in \Omega_F}(x_w,y_w)])=\psi_{H.E/E}(\tilde f_E[\prod_{w\in
\Omega_F}(x_w,y_w)])\\
&=\psi_{H/F}(\prod_{w\in\Omega_F} N_{E/F}(\tilde f_E[(x_w,y_w)]))=\psi_{H/F}(\prod_{w\in\Omega_F}\delta)=1,
\endaligned$$
the last equation holds by global class field theory. Therefore we only need to consider the condition in Proposition \ref{extra} which comes from $\Theta=E(\sqrt[4]{7})$ . Let $\Gal(\Theta/E)\cong \Z/4\Z=\{\pm 1,\pm {\bf i}\}$, where ${\bf i}$ is the primitive 4-th root of unity.

Suppose $p=2$. Since $\Theta/E$ is unramified, we have $$\psi_{\Theta/E}(\tilde f_E[(x_w,y_w)])=(-1)^{s_1}.$$

Suppose $p=7$. By the local solvable condition at place $7$, we have $s_2$ is even. Let $\xi\in \Z_7[\sqrt{-1}]^\times$ such that $N_{E_\frak l/F_w}(\xi)=\delta/(-14)^{s_2/2}$, where $\frak l$ (resp. $w$) is the unique place of $E$ (resp. $F$) over $7$. By Lemma \ref{computation}, we have $$\psi_{\Theta/E}(\tilde f_E[(x_w,y_w)])=(\sqrt{(-14)^{s_2/2}},7)_{4,\frak l}\cdot(\xi,7)_{{4,\frak l}}=1\cdot (-1)^{s_2/2}(\frac{a_1}{7}).$$

Suppose $(p,14\cdot N_{F/\Q}(\delta))=1$. Since $\Theta/E$ is unramified, we have $$\psi_{\Theta/E}(\tilde f_E[\prod_{w\in \Omega_F}(x_w,y_w)])=1.$$

Let $F'=\Q(\sqrt{-1})$ and $\Theta'=F'(\sqrt[4]{7})$. Note that $\Theta'/F'$ is abelian. By local class field theory, we have
$$\psi_{\Theta/E}(\tilde f_E[\prod_{w\mid p}(x_w,y_w)])=\psi_{\Theta'/F'}(N_{E/F'}(\prod_{w\mid p} \tilde f_E[(x_w,y_w)])).$$

Suppose $(\frac{-1}{p})=-1$. Then $7=u^2(-1)^t$, where $u\in \Z_p^\times$ and $t\in \Z_{\geq 0}$. Let $\frak P$ be the unique place of $F'$ over $p$. Therefore
$$\psi_{\Theta'/F',\frak P}(p)=(u\sqrt{(-1)^t}, p)_{\frak P}=(u^2,p)_p=1.$$

Now we assume that $p$ is a prime which splits into $\frak P_1$ and $\frak P_2$ over $F'$.

If $(\frac{7}{p})=-1$, then $\psi_{\Theta'/F',\frak P_i}(p)=\pm {\bf i}$ for $i=1 \text{ or } 2$. By global class field theory, $$\psi_{\Theta'/F',\frak P_1}(p)\cdot\psi_{\Theta'/F',\frak P_2}(p)=\psi_{\Theta'/F',v_1}(p)\cdot\psi_{\Theta'/F',v_2}(p)=1\cdot 1=1,$$
where $v_1$ and  $v_2$ are the unique place of $F'$ over $2$ and $7$ respectively.
Hence $$\psi_{\Theta'/F',\frak P_1}(p)=-\psi_{\Theta'/F',\frak P_2}(p)=\pm {\bf i}.$$

If $(\frac{7}{p})=1$ and $x^4 \equiv 7 \mod p$ is solvable,
then both $\frak P_1$ and $\frak P_2$ split in $\Theta/E$. Hence $\psi_{\Theta'/F',\frak P_i}(p)=1$ for $i=1 \text{ or }2$.

If $(\frac{7}{p})=1$ and $x^4 \equiv 7 \mod p$ is not
solvable, then both $\frak P_1$ and $\frak P_2$ are inert in
$\Theta/E$. Then $\psi_{\Theta'/F',\frak P_1}(p)=\psi_{\Theta'/F',\frak P_2}(p)=-1$.

Denote
$$\aligned &D_1'=\{p\in D(\delta):(\frac{-1}{p})=(\frac{14}{p})=1 \text{ and } (\frac{7}{p})_4=1\}\\
& D_2'=\{p\in D(\delta):(\frac{-1}{p})=-(\frac{14}{p})=1 \text{ and } (\frac{7}{p})=1\}.\endaligned$$
Note that $e_i$ is even for $p_i\in D_2\cup D_2'$ by the local condition at place $p_i$.
For any $(x_w,y_w)\in \bold X(o_{F_w})$, one has
$$\psi_{\Theta'/F'}(N_{E/F'}(\prod_{w\mid p_i}\tilde f_E[(x_w,y_w)]))= \begin{cases} 1 \ \ \ & \text{if
$p_i\in D_1'\cup D_2'$} \cr
 (-1)^{a_1}(-i_{\frak P_1})^{e_i}  \ \ \ & \text{if
$p_i\in D_1$} \cr (-1)^{\frac{1}{2}e_i} \ \ \ & \text{if
$p_i\in D_2$} \cr (-1)^{e_i} \ \ \ & \text{if $p_i\in D_3$}
\end{cases} $$ where $a_i=ord_{\frak P_i}(N_{E/F'}(x_{w}+y_{w}\sqrt{-1}))$,
$\frak P_i$ is a prime of $E$ above $p_i\in D_1$ and $i_{\frak P_i}=\psi_{\Theta'/F',\frak P_i}(p_i)=\pm {\bf i}$.

Suppose $D_1 =\emptyset$. By Proposition \ref{extra}, $\bold X(o_{F})\neq \emptyset$
if and only if
$$(\frac{a_1}{7})=(-1)^{s_1+s_2/2+\sum_{p_i\in D_2}\frac{e_i}{2}+\sum_{p_i\in
D_3}e_i}.$$

Now we suppose $D_1 \neq \emptyset$. Let $n'=p_1^{e_1}\cdots p_g^{e_g}$. We have $$1=(14,N_{F/\Q}(\delta))_7=(-1)^{s_2}\cdot (7,n')_7=(-1)^{s_2}(7,n')_2\prod_{i=1}^g(\frac{7}{p_i})^{e_i},$$ the last equation holds by the quadratic reciprocity law. By the local solvability at place $2$, $$1=(-1,N_{F/\Q}(\delta))_2=(-1)^{s_2}\cdot (-1,n')_2=(-1)^{s_2}\cdot (7,n')_2.$$
Therefore $\prod_{i=1}^g(\frac{7}{p_i})^{e_i}=1.$
If $(\frac{-1}{p_i})=-1$ or $(\frac{14}{p_i})=-1$ or $(\frac{-14}{p_i})=-1$, then $e_i$ is even by local conditions. Therefore we have
$$\sum_{p_i\in D_1}{e_i}\equiv 0 \mod 2.$$
Therefore the product of the image of $\psi_{\Theta'/F'}$ over $p\in D_1$ is $\{\pm 1\}$ when $D_1 \neq \emptyset$. Hence there exists $\prod_{w\in \Omega_F}(x_w,y_w)\in \prod_{w\in \Omega_F}\bold X(o_{F_w})$ such that $$\psi_{\Theta/E}(\prod_{w\in \Omega_F}\tilde f_E[(x_w,y_w)])=1.$$
The proof is complete.
\end{proof}


\bf{Acknowledgment} \it{ The authors would like to thank
Colliot-Th\'el\`ene for helpful comments on the early version of the
paper. The work is supported by the Morningside Center of
Mathematics. The first author is supported by NSFC, grant \#
10671104 and the second author is supported by NSFC, grant \#
10325105 and \# 10531060. }


\end{document}